\documentclass[amsthm,seceqn,secthm]{elsart} 

\usepackage[utf8x]{inputenc}
\usepackage[english]{babel} 
\usepackage{amsmath,amssymb,amsfonts}
\usepackage{enumerate}  
\usepackage{hyperref}  
                 \hypersetup{ pdfborder={0 0 0}, 
                              colorlinks=true, 
                              linkcolor=blue,
                              urlcolor=blue,
                              citecolor=blue,
                              linktoc=page,
                              pdfauthor={Giovanni Staglian{\`o}}, 
                              pdftitle={Examples of special quadratic birational transformations}} 
\usepackage{color,colortbl}
\usepackage{mathptmx}
\usepackage{natbib}
\usepackage{yjsco}
\usepackage{embedfile}
\embedfile{codes.jar}

\definecolor{Gray}{gray}{0.9}

\newcommand{\I}{{\mathcal{I}}}    
 
\renewcommand{\O}{{\mathcal{O}}}  
   
\newcommand{\E}{{\mathcal{E}}}   
\newcommand{\B}{{\mathfrak{B}}}
    
\newcommand{\QQ}{{\mathbb{Q}}}
\newcommand{\ZZ}{{\mathbb{Z}}}
\newcommand{\CC}{{\mathbb{C}}} 
\newcommand{\FF}{\mathbb{F}}     
\newcommand{\GG}{{\mathbb{G}}}   
\newcommand{\PP}{{\mathbb{P}}}

\begin{document}

\begin{frontmatter}
\title{Examples of special quadratic birational transformations into complete intersections of quadrics}
\author{Giovanni Staglian\`o}
\address{Instituto de Matem\'atica -- Universidade Federal Fluminense}
\date{\today}
\ead{\href{mailto:giovannistagliano@gmail.com}{giovannistagliano@gmail.com}}
\thanks{This work was supported by a BJT fellowship from CAPES (N. A028/2013).}
\begin{abstract}
In our previous works 
(\citeyear{note}, \citeyear{note2}), 
we provided a finite list of properties characterizing all potential types of
 quadratic birational transformations of a projective space 
into a factorial variety, whose base locus is smooth and irreducible.
However, some existence problems remained open.
Among them one had to prove that the image of a given transformation
 was factorial, but in two particular situations,  
 even the mere existence of the transformation 
 was left as an open problem.
In this paper, we use computer algebra methods to construct explicitly
four  examples of such transformations; two of them were among those for which 
it was not known that the image was factorial, and the other two show the existence of the two   
above referred transformations. 
\end{abstract}
\end{frontmatter}

\section{Introduction}
A birational transformation $\varphi:\PP^n\dashrightarrow Z\subseteq\PP^N$
of a projective space into a 
nondegenerate, linearly normal, factorial variety  $Z\subseteq \PP^N$ 
is called \emph{special} if its base locus $\B\subset\PP^n$ 
is smooth and irreducible 
(here, a factorial variety is a variety all of whose local rings are UFD).
In this situation, one has that  
the support of the singular locus of $Z$  
is contained in the support of the base locus 
 of $\varphi^{-1}$; 
 we also always assume that
  this inclusion is strict, i.e. that $Z$ is ``not too singular''.
The special transformation $\varphi$ is said to be 
of type  $(d_1,d_2)$ if $\varphi$ and $\varphi^{-1}$ are, respectively, defined 
by linear subsystems of $|\O_{\PP^n}(d_1)|$ and  $|\O_{Z}(d_2)|$; 
$\varphi$ is a \emph{quadratic} transformation if $d_1=2$.
The case where $Z\subseteq\PP^N$ is a factorial complete intersection 
is particularly interesting, which, in a sense, generalizes the classical notion of
special Cremona transformations, that is 
$Z=\PP^N$.  

The study of special birational transformations  
with some fixed numerical invariants (for example, 
when fixing the dimension of the base locus, or fixing the type)
is a classical problem and often challenging.  
The first general results were obtained by 
\citet[see also \citeyear{crauder-katz-1991}]{crauder-katz-1989}, who classified 
 all special Cremona transformations whose base locus has dimension at
most two. Soon after, \citet{ein-shepherdbarron} classified special Cremona 
transformations of type $(2,2)$   
as those given by systems of quadrics through Severi varieties. 
More recently, 
in the same spirit, we classified  in (\citeyear{note}) 
 special birational 
transformations of type $(2,2)$ into smooth quadric hypersurfaces as 
those given by systems of quadrics through hyperplane sections of Severi varieties. 
This classification has been extended in two directions: 
by \cite{alzati-sierra-quadratic} when the image $Z$ is an $LQEL$-manifold, and  
by \cite{li-special} when $Z$ is a smooth hypersurface.
In a different direction, \citet{alzati-sierra} 
extended the classification of \citeauthor{crauder-katz-1989} to the case 
when $Z\subseteq\PP^N$ is a prime Fano manifold.
They have also obtained classification results 
in the case when the dimension of the base locus is three, 
except that they have not treated two cases: 
that of quadratic transformations of $\PP^8$, and that of
cubic transformations of $\PP^6$; see \citet[Remark~9]{alzati-sierra} and
\citet[Corollary~1]{crauder-katz-1991}.

Motivated by this and 
some related partial results in \citet{note}, 
 we studied in (\citeyear{note2}) special quadratic transformations 
 into  factorial varieties with $\dim\B\leq3$.
A finite list
of all potential cases was provided.
The main difficulty was to deal with the case of
transformations of $\PP^8$ of type $(2,d)$, $d>1$,
having a nondegenerate threefold as base locus;
specially, it was difficult to exhibit examples  of these transformations.
In some cases,
we had an example of  transformation, but 
we were not able to establish that the image was factorial,  
due to lack of explicit equations.
However, in two cases, 
even the mere existence of the transformation was left as an open problem.
These cases were the following:
\begin{enumerate}[(I)]
\item\label{EsA} a special Cremona transformation of type $(2,5)$
with base locus a threefold scroll of degree $12$ 
over a rational surface $Y$ with $K_{Y}^2=5$; 
\item\label{EsB} a special birational transformation 
of type $(2,4)$ into a quadric hypersurface with base locus
a threefold scroll of degree $11$  over the Hirzebruch surface $\FF_1$.
\end{enumerate}
In the latter case,
the existence of such a scroll over $\FF_1$
was known thanks to \citet{alzati-fania-ruled}, 
but we did not know, for example, if it was cut out by quadrics.

By establishing the existence for case (\ref{EsA}), one also solves 
classification problems about special Cremona transformations.
Indeed, besides a known example due to \citet{hulek-katz-schreyer}, 
the transformation in (\ref{EsA}) can be the only possible 
special quadratic Cremona transformation
having three-dimensional base locus.
Furthermore, 
the special quadratic Cremona transformations having three-dimensional base locus 
are the same as the special Cremona transformations 
of type $(2,5)$. 
Similarly,
by establishing the existence for case  (\ref{EsB}), 
one solves 
classification problems about transformations into hypersurfaces.
See Subsection 
 \ref{subsec: applications}
for more details.

The main result of this paper, Theorem \ref{theoremscrolls}, states the 
existence of four linked special quadratic birational transformations of $\PP^8$
belonging to the 
aforementioned
list   in \cite{note2}.
Two of them are the
sought examples (\ref{EsA}) and (\ref{EsB}), 
while the other two 
were already known, 
but without knowing whether the image was factorial;
we establish the factoriality for these cases 
(see Subsection \ref{subsec: applications} and Table \ref{table: class}).

In the construction of one of the four transformations, incidentally 
we encounter an example of ``almost special'' quadro-quadric 
Cremona transformation of $\PP^{11}$; 
in Section \ref{sec: cremona} we study this example in details. 
Inspired by this, 
we explicitly construct  a quadro-quadric Cremona transformation of $\PP^{20}$, 
which 
should be 
related 
to the rank 3 
         Jordan algebra considered by 
\citet[Examples~5.8 and 5.11]{pirio-russo-cremona}. 
 
 In order to construct all our examples of quadratic transformations, 
 we proceed first to construct 
  a (candidate) restriction map $\psi$ to a general linear subspace
of a certain dimension,  
and then we try to lift $\psi$ 
via a birational map $\widetilde{\psi}$. 
In the cases in which we expect 
 that the base locus 
is a threefold scroll over a surface, 
a well-known result of adjunction theory 
suggests a way to cook up an explicit parameterization of a
  hyperplane section  $S$ of such a scroll; 
 then we  define $\psi$
 by the linear system of quadrics passing through $S$.
Once we have obtained a birational map $\psi$, 
we proceed as follows.
We first determine the ideal $\mathfrak{a}$ 
of the image of $\psi$ and the inverse map $\psi^{-1}$. 
Then, 
by combining the forms defining $\psi^{-1}$
with some generators of $\mathfrak{a}$,
we define a suitable rational map 
that in some way resembles the inverse of the hypothetical map $\widetilde{\psi}$.
We then attempt to invert this map.

The main difficulty is that,
working in characteristic zero and in  large enough projective spaces,  
our initial  maps $\psi$'s already 
involve polynomials 
with huge coefficients and many terms. 
But the 
general standard algorithms 
 often yield  massive amounts of data, 
even with rather simple input data. 
So it is necessary to use
some ad-hoc algorithms  
that allow to reduce the use of data and calculation times. 
In Section \ref{sec: methods}, we describe these algorithms as well as
 some computational tools, 
 which allow us to prove 
 rigorously  why our (human unreadable)
 output data provide the new examples of 
 transformations.
 
 All  
 codes 
described in this paper 
 are 
 written and executed in 
 {\sc Macaulay2}  \citep{macaulay2}, and
all the input and output files  
are located at \url{http://goo.gl/eT4rCR} 
  (an archive containing them all is also attached to the pdf version of this paper). 
The implementations of 
the algorithms in Section \ref{sec: methods} are contained in the file 
\verb!bir.m2!.    

\section{Some computational methods for rational maps between projective spaces}\label{sec: methods}
In this section, we briefly describe the algorithms 
that we will use in the following.   

\subsection{Computing inverse map}\label{russo-simis} 
The most important tool  
we need is 
an efficient way to compute inverses of birational maps 
with source a projective space.
For this,
we compute the inverse of 
a birational map onto its image,
 $\varphi:\PP^n\dashrightarrow\PP^m$, 
using the algorithm described by \citet{russo-simis}; 
although their approach  is not general, 
it works for our examples. 
More explicitly, we use the  following:
\begin{alg}[\texttt{invertBirationalMapRS}]\label{invertBirationalMapRS-algo} 
 \begin{description}\hspace{1pt}
  \item[\texttt{Input:}]  a row matrix $F=(F_0,\ldots,F_m)$ over the polynomial ring 
$k[x_0,\ldots,x_n]$
of the source space,  
representing  the map $\varphi$, 
and an ideal $\mathfrak{a}$ of the polynomial ring $k[y_0,\ldots,y_m]$
of the target space, representing the ideal of the image of $\varphi$.
  \item[\texttt{Output:}] a row matrix $G=(G_0,\ldots,G_n)$ over the polynomial ring
 $k[y_0,\ldots,y_m]$,  
representing (a candidate for)  a lifting to $\PP^m$ of $\varphi^{-1}$.
 \end{description}
\begin{itemize}
 \item   Find
 generators $\{(L_{0,j},\ldots,L_{m,j})\}_{j=1,\ldots,q}$ for   
  the module of linear syzygies of 
  $(F_0,\ldots,F_m)$;
  \item compute the Jacobian matrix $\Theta$ 
of the biforms $\{\sum_{i=0}^m y_i\,L_{i,j} \}_{j=1,\ldots,q}$ 
with  respect to the variables $x_0,\ldots,x_n$
(note that all the 
entries of $\Theta$ are linear forms lying 
in the ring $k[y_0,\ldots,y_m]$), and consider
the map of graded free modules 
$(\Theta\ \mathrm{mod}\ \mathfrak{a}):
\left(k[y_0,\ldots,y_m]/\mathfrak{a}\right)^{n+1}\rightarrow 
\left(k[y_0,\ldots,y_m]/\mathfrak{a}\right)^{q}$;
\item return 
a generator $G$ for the kernel of $(\Theta\ \mathrm{mod}\ \mathfrak{a})$,
seeing it as a row matrix 
over $k[y_0,\ldots,y_m]$. 
\end{itemize}
\end{alg}
While computing the inverse of a map $\varphi$ as above 
requires non-trivial algorithms,
checking that a given map $\eta:\PP^m\dashrightarrow\PP^n$ 
is a lifting of 
the inverse map of $\varphi$ is straightforward. 
\begin{alg}[\texttt{isInverseMap}]\label{isInverseMap-algo} 
 \begin{description}\hspace{1pt}
  \item[\texttt{Input:}] two row matrices representing 
  the maps  $\varphi$ and $\eta$.
 \item[\texttt{Output:}] boolean value according to the condition that $\eta\circ\varphi$ 
 coincides 
 with the identity of $\PP^n$.
 \end{description}
 \begin{itemize}
  \item Determine a row matrix $(H_0,\ldots,H_n)$ of forms defining  the composition 
  $\eta\circ\varphi:\PP^n\dashrightarrow\PP^n$;
  \item put $T=H_0/x_0$, and return \texttt{true} if and only if
  $(H_0,\ldots,H_n)=(x_0\,T,\ldots,x_n\,T)$.
 \end{itemize}
 \end{alg}

\subsection{Checking birationality} \label{subsec: birationality}
One can check the birationality of a given map without going
through calculation of the inverse map.
This can be done via computing the degree.
 Indeed,
if $\varphi: 
\PP^n\dashrightarrow 
\PP^m$
is a rational map
defined by a linear system $\langle F_0,\ldots,F_m\rangle$,
we consider the corresponding affine map
$\varphi_a:\mathbb{A}^{n+1}
\rightarrow \mathbb{A}^{m+1}
$,
and the induced ring map 
$\widetilde{\varphi}:k[y_0,\ldots,y_m]\rightarrow k[x_0,\ldots,x_n]$.
If $\mathfrak{p}\subset k[x_0,\ldots,x_n]$ is the homogeneous ideal 
of a (general) point $p\in\PP^n$, then 
we have 
that 
$\overline{\varphi^{-1}\left(\varphi(p)\right)}$ 
is the projectivized of the cone
$\overline{ {\varphi_a}^{-1}\left(\varphi_a\left(V(\mathfrak{p})\right)\right) 
\setminus V(F_0,\ldots,F_m) }$, and this is defined 
by the saturation ideal 
$$\varphi_{\mathfrak{p}}:=\left(\widetilde{\varphi}\left({\widetilde{\varphi}}^{-1}(\mathfrak{p})\right)\right):
{\left(F_0,\ldots,F_m\right)}^{\infty} .$$
So we can compute the degree of $\varphi$ 
by computing the degree and the dimension of $V(\varphi_{\mathfrak{p}})\subset\PP^n$.
That is, we have the following: 
\begin{alg}[\texttt{degreeOfRationalMap}]\label{degreeOfRationalMap-algo} 
 \begin{description} Probabilistic approach. 
  \item[\texttt{Input:}] row matrix representing a rational map $\varphi:\PP^n\dashrightarrow\PP^m$.
 \item[\texttt{Output:}] the degree of $\varphi$.
 \end{description}
 \begin{itemize}
  \item Pick up a random point $p=V(\mathfrak{p})$ on $\PP^n$ (i.e., take $n$ random linear forms on $\PP^n$);
  \item compute the above defined ideal $\varphi_{\mathfrak{p}}$;
  \item if $\dim V(\varphi_{\mathfrak{p}})>0$ return $0$, else return $\deg V(\varphi_{\mathfrak{p}})$.
 \end{itemize}
 \end{alg}
The strength of this algorithm  is that {\sc Macaulay2}  
computes quotient and saturation of ideals generally without difficulty.  
The approach is probabilistic because 
 we could not get the right answer when 
 $p$ 
 belongs to a certain closed subset $C\subsetneq\PP^n$.
One can  get a non-probabilistic version, by 
performing the computation with a 
generic (i.e., symbolic) point $p$,
but this seems very slow.
\begin{rem}\label{remark: projectiveDegrees}
Note that
Algorithm \ref{degreeOfRationalMap-algo}
can be 
adapted in an obvious way to obtain 
a probabilistic algorithm 
for the calculation of all projective degrees of a rational map, 
not only from a projective space 
\citep[for the definitions see][Example~19.4]{harris-firstcourse}.
We have roughly implemented this with
     the routine \verb!projectiveDegrees!. 
A non-probabilistic (but more expensive) 
way to do this is described  
by \citet{aluffi}. 
\end{rem}
\subsection{Computing ideal of the image}\label{quadricsInImage} 
Another  useful remark is that we
can
 compute the ideal of the image of a rational map 
with source a projective space,
without going
through calculation of the kernel of the corresponding ring map,
at least when we know that the ideal 
is generated in degree less than some small integer.
We can proceed as in the following: 
\begin{alg}[\texttt{homogPartOfImage}]\label{homogPartOfImage-algo} 
 \begin{description}\hspace{1pt}
  \item[\texttt{Input:}] the corresponding ring map 
of a rational map 
$\varphi=
[F_0,\ldots,F_m]:\PP^n\dashrightarrow\PP^m$, 
and 
an integer $d$.
 \item[\texttt{Output:}] a basis of the vector space  $H^0(\PP^m,\mathfrak{I}(d))$,
where $\mathfrak{I}$ denotes the ideal sheaf of the 
image.
 \end{description}
 \begin{itemize}
  \item Take a generic homogeneous degree $d$ polynomial 
$G(y_0,\ldots,y_m)=\sum_{|I|=d}a_I\mathbf{y}^I$;
\item substitute 
the $y_i$'s with the $F_i$'s, so obtaining a homogeneous polynomial $\widetilde{G}$ 
in the ring of $\PP^n$, which
 vanishes identically if and only if $G$ lies in $H^0(\PP^m,\mathfrak{I}(d))$;
 \item solve the  
  homogeneous linear system in the $a_I$'s given by the coefficients 
 of $\widetilde{G}$, finding a basis of solutions;
 \item return the homogeneous polynomials
 in $y_0,\ldots,y_m$
 corresponding to this basis,
 by reinterpreting the elements of the basis
  as lists of coefficients of $G$.
 \end{itemize}
 \end{alg}
 This method may turn out to be much faster than computing  
the kernel 
of the corresponding ring map. 
This is for example the case for the maps of Theorem \ref{theoremscrolls}, 
see Subsection~\ref{Subsection appendix Image phij}. 

\subsection{Computing base locus of the inverse map}\label{baselocusinversa} 
For birational maps, like those returned by 
Algorithm \ref{invertBirationalMapRS-algo},
we can determine the base locus of the inverse map
via a heuristic method, without computing the inverse map. Indeed, 
consider the special case, the only one we are interested in, in which we have a
rational map  $\eta:\PP^N\dashrightarrow\PP^n$ 
 defined by a linear system of hypersurfaces of degree $d$, and 
 $Z\subset\PP^N$ is an $n$-dimensional subvariety such that the restriction 
$\eta|_{Z}:Z\dashrightarrow\PP^n$ is a birational map of type $(d,2)$. 
If $H\subset\PP^N$ is a general hyperplane, then 
the image of $H\cap Z$ via $\eta$ is a general member of the linear system defining 
the inverse map $\varphi={\eta|_{Z}}^{-1}$.
Thus, by considering the images of $N+1$ general hyperplane sections 
$H_i\cap Z\subset\PP^N$, $i=0,\ldots, N$, 
we obtain $N+1$ quadric hypersurfaces $Q_i\subset\PP^n$ such that the  
intersection $\bigcap_{i=0}^N Q_i$ is the base locus $\B\subset\PP^n$ of $\varphi$.
Even better, one can  
compute just the images 
of linear sections obtained by intersecting $Z$ with 
general subspaces of higher codimension $c>1$. So, 
when $c$ 
is not too high (more precisely, when $c$
 is less than the codimension  of $\B$),
one determines 
intersections $\bigcap_{j=1}^{c} Q_j$ of $c$ general members 
of the linear system defining $\varphi$ and can obtain $\B$ by intersecting them.
These considerations yield the following: 
\begin{alg}[\texttt{baseLocusOfInverseMap}]\label{baseLocusOfInverseMap-algo} 
 \begin{description}\hspace{1pt}
  \item[\texttt{Input:}] an ideal representing the subvariety $Z\subset \PP^N$,
a ring map representing the rational map $\eta:\PP^N\dashrightarrow \PP^n$,
and a list of ideals representing a sufficiently large set of sufficiently 
    general linear subspaces of $\PP^N$.
 \item[\texttt{Output:}] an ideal, candidate to be 
the ideal of the base locus $\B$ of $\varphi$. 
 \end{description}
 \begin{itemize}
\item For each given subspace $L\subset\PP^N$ compute a basis of 
 $H^0(\PP^n,\mathfrak{I}(2))$,
where $\mathfrak{I}$ denotes the ideal sheaf of the 
image of the  restriction map 
$\eta|_{Z\cap L}:Z\cap L\dashrightarrow\PP^n$ (use Algorithm \ref{homogPartOfImage-algo} if $Z=\PP^N$); 
\item return the ideal generated by all these quadratic forms.
\end{itemize}
\end{alg}

 \subsection{Checking smoothness of the base locus}\label{smoothnessBaseLocus}
We can check if a 
closed 
equidimensional 
subscheme 
$V(I)\subset\PP^n_{\QQ}$ 
is smooth over $\mathbb{Q}$ 
(equivalently, over  the algebraic closure of $\mathbb{Q}$) 
by reducing to prime characteristic.

More precisely,
if
$I$
is 
generated by forms $F_0,\ldots,F_m\in\ZZ[x_0,\ldots,x_n]$,  
consider the  proper morphism $f:\mathrm{Proj}\left(\ZZ[x_0,\ldots,x_n]/I\right)
\rightarrow\mathrm{Spec}(\ZZ)$.   
From  Chevalley's upper semicontinuity theorem \citep[IV~13.1.5]{EGA}, 
  for all but finitely many 
 closed points $\mathfrak{p}\in\mathrm{Spec}(\ZZ)$, we
 have $\dim f^{-1}(\mathfrak{p})=\dim f^{-1}(\mathfrak{0})$, where $\mathfrak{0}$
 denotes the generic point of $\mathrm{Spec}(\ZZ)$.  
 Let 
 $c=n-\dim f^{-1}(\mathfrak{0})$ and
denote by $J\subset \ZZ[x_0,\ldots,x_n]$ 
the ideal generated by all 
minors of order $c$ of the Jacobian matrix $\left(\partial F_i/\partial x_j \right)_{i,j}$
together with the forms $F_i$'s. 
 Again we have a proper morphism 
 $g:\mathrm{Proj}\left(\ZZ[x_0,\ldots,x_n]/J\right)
\rightarrow\mathrm{Spec}(\ZZ)$, and we have that 
$\dim g^{-1}(\mathfrak{0})= -1$ 
if and only if $f^{-1}(\mathfrak{0})$ is a smooth scheme.
So once again from Chevalley's upper semicontinuity theorem, it follows that 
if for some $\mathfrak{p}\in\mathrm{Spec}(\ZZ)$ we have
$\dim g^{-1}(\mathfrak{p})= -1$, then we have this 
for all but finitely many $\mathfrak{p}\in\mathrm{Spec}(\ZZ)$, 
and $f^{-1}(\mathfrak{0})$ is a smooth scheme. 

Thus we have the following:
\begin{alg}[\texttt{isSmooth}]\label{isSmooth-algo} 
 \begin{description}\hspace{1pt}
  \item[\texttt{Input:}] an ideal 
$I=(F_0,\ldots,F_m)\subset \ZZ[x_0,\ldots,x_n]$ as above, 
and a prime number  
$p$.
 \item[\texttt{Output:}] a boolean value according to the condition of being smooth   
for  
$V(I)\subset \PP^n_{\QQ}$.
 \end{description}
 \begin{itemize}
  \item Compute the codimension $c$ of the scheme $V(I)$ over $\QQ$;
  \item replace $I$ with $(I\,\mathrm{mod}\ p)$;
  \item compute the ideal $J$ generated by all 
minors of order $c$ of the Jacobian matrix $\left(\partial F_i/\partial x_j \right)_{i,j}$
together with the forms $F_i$'s;
\item compute the codimension $e$ of $V(J)$ (as scheme over $\ZZ/(p)$);
\item if $e>n$ then return \texttt{true}, 
else (print a message with a suggestion to) choose another prime $p$.  
 \end{itemize}
 \end{alg}
\begin{rem}
 With our implementation of Algorithm \ref{isSmooth-algo}, we 
 can pass the optional argument  \verb!Use=>Sage! to transfer the computation to
the software system Sage \citep{sagemath}, because it seems that the latter 
is much faster to perform this type of computations.
We also pass the prime $p$ via the optional argument \verb!reduceToChar=>p!.
\end{rem}

\subsection{Checking irreducibility of the base locus}\label{absoluteConnection}
All our objects will be defined 
over the rational field 
$\QQ$,\footnote{%
Actually, all our objects will be defined 
over $\mathbb{Z}$, and one may think them defined over $\ZZ/(p)$, where $p$ is a 
prime.}
where we can do the exact calculations,
but we will be interested in their properties over the 
complex field $\CC$. Of course, the property of irreducibility does not transfer from $\QQ$ 
to $\CC$. However, 
one can check if a smooth 
subscheme $X\subset \PP^n_{\QQ}$  is irreducible over $\CC$,
by computing the number of connected components   
 over the algebraic closure of $\QQ$, that is   
   the dimension over $\QQ$ of $H^0(X,\O_X)$.
   Fortunately, {\sc Macaulay2} is well able to perform these calculations,
   thereby making trivial the 
 implementation of the following:
 \begin{alg}[\texttt{numberConnectedComponents}]\label{numberConnectedComponents-algo} 
 \begin{description}\hspace{1pt}
  \item[\texttt{Input:}] homogeneous ideal $I$ 
  of, for example, a smooth subscheme $X\subset\PP^n$ over $\QQ$;
 \item[\texttt{Output:}] the number of connected components of $X$ over $\overline{\QQ}$.
 \end{description}
 \begin{itemize}
  \item Compute the $\QQ$-vector space $H^0(X,\O_X)$ and return its dimension.
 \end{itemize}
 \end{alg}
 \subsection{Checking factoriality of the image}\label{subsec: factoriality}
For the special birational transformations we consider in this paper,
checking that
the image of the transformation is a factorial variety
is reduced to a simple computational verification. Indeed,
  from Grothendieck's parafactoriality theorem 
 \citep[see][XI, 3.14]{sga2},
 an $n$-dimensional local complete intersection  is factorial whenever
 the dimension of its singular locus 
 is less than $n-3$.
 Of course, complete intersections 
are also local complete intersections.

Furthermore, it is worth observing that
  cones over 
  Grassmannians (embedded by their Pl\"ucker embedding) are always factorial varieties.
  This follows from the following more general result.    
 \begin{prop}\label{proposition: factoriality} Let $Z\subset\PP^n$ be a smooth projectively normal variety 
with $\mathrm{Pic}\,Z\simeq\mathbb{Z}\langle \mathcal{O}_Z(1)\rangle$.
Denote by $Z_i\subset\PP^{n+i+1}$ the cone over $Z$ of vertex a linear 
space $L_i\subset \PP^{n+i+1}\setminus \PP^n$ of dimension $i$ ($Z=Z_{-1}$); in other words,
 $Z_i$ is  the classical projective cone over $Z_{i-1}\subset \PP^{n+i}$ in $\PP^{n+i+1}$.
Then $Z_i$ is a factorial variety. 
\end{prop}
\begin{proof}\citep[See also][Exercise~II~6.3]{hartshorne-ag}. 
Denote by $S(Z)$ and $S(Z_i)$ the homogeneous coordinate rings of, respectively, $Z$ and $Z_i$.
The conditions on $Z$ imply the fact that $S(Z)$ is a UFD.
Since $S(Z_i)$ is isomorphic to
the polynomial ring in one variable with coefficients in $S(Z_{i-1})$,
we see by induction that, for every $i$, $S(Z_i)$ is a UFD.
Now let $p\in Z_i$ and let $\mathcal{O}_{p,Z_i}$ denote the local ring 
of $p$ on $Z_i$. If $p\notin L_i$, then $p$ is a smooth point of $Z_i$ and  
so $\mathcal{O}_{p,Z_i}$ is a UFD. 
If $p\in L_i$,  then $\mathcal{O}_{p,Z_i}$  is isomorphic to  the localization of $S(Z_{i-1})$ 
 at the maximal ideal generated by all homogeneous elements of positive degree, 
 and so, again, it is a UFD.
\end{proof}
\section{Transformations coming from threefold scrolls over surfaces}\label{sec: scrolls}
In this section we prove our main result:
 \begin{thm}\label{theoremscrolls}
  Let $j=0,1,2,3$. There exists a special birational transformation 
$\varphi_j$ of type $(2,2+j)$ 
from $\PP^8$ into a factorial complete intersection of $3-j$ quadrics in $\PP^{11-j}$, 
whose base locus 
is  a  linearly normal threefold scroll $\B_j=\PP_{Y_j}(\E_j)\subset\PP^8$ of degree $9+j$ 
 and sectional genus $3+j$ 
over a surface $Y_{j}$, 
where $Y_0=\PP^2$, $Y_1=Q^2$ is the quadric surface, $Y_2=\FF_1$ is the Hirzebruch surface, 
and 
 $Y_3$ is a rational surface with $K^2=5$.
 \end{thm}
We prove Theorem \ref{theoremscrolls} providing  
an algorithm producing the explicit 
rational maps $\varphi_j$'s, 
 and then we will can show everything 
 reducing to computational verifications. 
The algorithm has been implemented in {\sc Macaulay2}, 
dividing the input code into
three 
files, \verb!INPUT1!, \verb!INPUT2!, \verb!INPUT3!, 
together 
with the file named \verb!bir.m2!. 
From each input it comes out a respective output file
and 
the final, 
 named \verb!equationsOfBj.m2!,
 contains 
the maps, the 
inverses, and
the equations of the images. 
\subsection{Construction of the transformations}\label{sec construction} 
\subsubsection{Construction of restriction maps}
In this subsubsection, we describe the algorithm implemented in the 
 file \verb!INPUT1!, which determines 
the equations of smooth surfaces, good candidates 
to be general hyperplane sections of the desired scrolls.
In other words, 
we describe how to construct, for $j=0,1,2,3$, a smooth surface $S_j\subset\PP^7$  
of degree $9+j$ and cut out by $12-j$ quadrics. 
Running this code on our machine requires a few minutes.

To begin with, we recall the following well-known result of adjunction theory 
\citep[see for example][]{beltrametti-sommese}. 
\begin{fact}\label{fact-beltrametti-sommesse}
Let $X=\PP_Y(\E)$ be a projective smooth threefold, embedded as a linear scroll 
over a smooth surface $Y$, where $\E$ is a rank $2$ vector bundle on $Y$, 
and let $p:X\rightarrow Y$ denote the structural morphism.
If 
$S$ is a smooth hyperplane section of $X$, 
then 
the restriction $p|_S$ of $p$ to $S$ is the reduction map of $(S,H|_S)$ 
  to $(Y,\det\E)$, where $\det\E$ is very ample and the number 
  of points blown up by $p|_S$ 
  is $c_2(\E)$. 
\end{fact}
From this it follows that if  
the scroll $\B_j=\PP_{Y_j}(\E_j)\subset\PP^8$ 
in the statement of Theorem \ref{theoremscrolls} exists, then 
the surface $Y_j$ is embedded in $\PP^{7+c_2(\E_j)}$ by $\det\E_j$ as a surface 
of degree $9+j+c_2(\E_j)$,
and it is the reduction of 
the general 
hyperplane section $S_j\subset\PP^7$ of $\B_j$.  
But one can also deduce that, putting $s_j=c_2(\E_j)$, necessarily we must have 
 $s_0=7$, $s_1=8$, $s_2=10$, $s_3=8$; so we do not have much choice for the
 embedding of $Y_j$ and 
 construct 
 one of its parameterizations  $\nu_j$
 as follows: 
\begin{itemize}
\item For $j=0$, we simply take the $4$-tuple embedding $\nu_0$ of $\PP^2$ in $\PP^{14}$ 
as a surface of degree $16$.
\item For $j=1$, we take a parameterization of a smooth quadric surface 
$\PP^2\dashrightarrow Q^2\subset\PP^3$ and compose it with the $3$-tuple embedding 
of $\PP^3$ in $\PP^{19}$; so we obtain a birational map
$\nu_1:\PP^2\dashrightarrow\PP^{19}$ parameterizing 
a surface of degree $18$ 
in $\PP^{15}\subset\PP^{19}$.
\item For $j=2$, we take a parameterization $\nu_2:\PP^2\dashrightarrow\PP^{17}$ 
of the Hirzebruch surface 
$\FF_1$ embedded in $\PP^{17}$ by $3C_0+5f$ as a surface of degree $21$;
this is given by the linear system
of all quintic curves containing a triple point in $\PP^2$.
\item For $j=3$, we take a parameterization  of the del Pezzo surface 
of degree $5$ in $\PP^5$, which is 
given by the linear system 
of cubic curves passing through $4$ points in general position in $\PP^2$,
and compose it with the $2$-tuple embedding 
of $\PP^5$ in $\PP^{20}$; so we obtain a birational 
map $\nu_3:\PP^2\dashrightarrow\PP^{20}$ parameterizing 
a surface of degree $20$ 
in $\PP^{15}\subset\PP^{20}$.
\end{itemize}
The next step is to project these surfaces from $s_j$ sufficiently general points. 
That is, for $j=0,1,2,3$,
we consider the composition of the parameterization 
$\nu_j$ 
with a sequence of $s_j$ general inner projections.
Equivalently, we pick up $s_j$ general points $p_{j}^1,\ldots,p_{j}^{s_j}$ on the source space 
of $\nu_j$ 
and compose $\nu_j$ with the restriction to 
$Y_j\subset \PP^{7+s_j}$ of the projection from the linear span 
$\langle \nu_j(p_{j}^1),\ldots ,\nu_j(p_{j}^{s_j}) \rangle\subset\PP^{7+s_j}$.  
Once this is done, we obtain a map   $\psi_j:\PP^2\dashrightarrow\PP^7$ 
and determine its image $S_j\subset\PP^7$. 
If
the picked up points $p_{j}^1,\ldots,p_{j}^{s_j}$ were  sufficiently general,
we now have a smooth (irreducible) surface $S_j\subset\PP^7$   
of degree $9+j$,
sectional genus $3+j$,
and cut out by $12-j$ quadrics. 
This is easily verifiable if one uses Algorithm \ref{isSmooth-algo} for 
the smoothness condition.

As 
different choices of points $p_{j}^1,\ldots,p_{j}^{s_j}$ 
  yield in general different $S_j$'s, 
 we have chosen them in order to obtain, as far as possible,
 more manageable coefficients in the equations of $S_j$.
 The ideals  
 of our four $S_j$'s are stored in the file \verb!equationsOfSj.m2!,
\subsubsection{Inverses of the restriction maps}
Now, for each $j=0,1,2,3$, 
 consider the rational map $\varphi'_j:\PP^7\dashrightarrow\PP^{11-j}$ 
defined 
by the linear system of quadrics passing through $S_j\subset\PP^7$.
Using Algorithm \ref{degreeOfRationalMap-algo},
anyone can easily convince himself that  $\varphi'_j$
 is birational onto its image, 
but we now determine the explicit inverse map. 
This is implemented in the file \verb!INPUT2!, and 
running this code on our machine requires about $17$ hours.

First of all, we must determine the ideal of the image of $\varphi'_j$.
We note that if $\varphi'_j$ is really what we expect, i.e.  
the restriction to a general $\PP^7\subset\PP^8$
of the map $\varphi_j$ in the statement of Theorem \ref{theoremscrolls},
then the image of $\varphi'_j$ 
must be a (complete)
intersection $Z_j\cap D_j\subset\PP^{11-j}$ 
of a complete intersection $Z_j$ of $3-j$ quadrics (i.e. the image of $\varphi_j$) 
and a  hypersurface $D_j$ of degree $2+j$. 
So,
we are only interested in finding 
generators for the ideal of the image of $\varphi'_j$
 of degree $\leq 2+j$.
For this, one can use Algorithm \ref{homogPartOfImage-algo};
however, when $j=2,3$,
it requires a lot of memory usage, and  we opt to compute the kernel of 
the corresponding ring map of $\varphi'_j$ 
(this is the most expensive part of the whole process). 
Having done this, we see that the achieved subscheme of $\PP^{11-j}$ 
is the expected complete
intersection $Z_j\cap D_j\subset\PP^{11-j}$, and can now 
apply Algorithm \ref{invertBirationalMapRS-algo}
in order to determine the inverse of $\varphi'_j$.
It returns 
$8$ forms on $\PP^{11-j}$ of degree $2+j$,
and we define the rational map 
$\eta_j:Z_j\dashrightarrow  \PP^8$ given by the linear system generated 
by these $8$ forms together with the form defining $D_j$.
We have stored our data 
$Z_j$ and $\eta_j$  in the file 
\verb!inversesOfRestrictions.m2!.

\subsubsection{Liftings of the restriction maps}\label{subsec: liftings}
In this subsubsection, 
we describe the algorithm implemented in the  
last input file, \verb!INPUT3!, 
which has the only purpose of determining (at least modulo projective 
transformations) the inverse map $\varphi_j$
of $\eta_j$, for $j=0,1,2,3$.
Running this code on our machine requires  about $16$ hours.
We proceed by discussing the cases according to the values ​​of $j$.

\paragraph*{} For $j=3$, we compute the inverse of $\eta_3$ 
via Algorithm \ref{invertBirationalMapRS-algo}.
One may also do this via Algorithm~\ref{baseLocusOfInverseMap-algo}.
Indeed the quadratic forms generating the ideal returned by 
Algorithm~\ref{baseLocusOfInverseMap-algo} define 
a map $\phi:\PP^8\dashrightarrow\PP^8$ 
such that $\phi\circ\eta_3$ 
is a projective transformation, and one can
compute 
$(\phi\circ\eta_3)^{-1}\circ \phi$.

\paragraph*{}
For $j=2$,  
we reduce 
the problem to that of determining the inverse of a 
Cremona transformation $\beta$ of $\PP^8$ of type $(8,2)$. 
This is done by parameterizing the quadric hypersurface $Z_2$ 
with a birational transformation $\rho$ of type $(2,1)$ of $\PP^8$; the map 
$\rho$ is nothing more than the inverse of the projection map of $Z_2$ 
from one of its smooth points.
Now, inverting $\beta$ still seems very difficult 
(even with Algorithm \ref{invertBirationalMapRS-algo}),
but we can nevertheless determine
 the base locus of the 
inverse map of $\beta$ 
by using Algorithm \ref{baseLocusOfInverseMap-algo} (which calls 
Algorithm \ref{homogPartOfImage-algo}). 
So, we obtain a map $\alpha:\PP^8\dashrightarrow\PP^8$ such that 
$\alpha\circ\beta$ is a projective transformation. 
At this point, if $\alpha\circ\beta$ is not the identity map, 
we compute its inverse and replace 
  $\alpha$ 
with $(\alpha\circ\beta)^{-1}\circ\alpha$. 
Finally, we compute the composition $\varphi_2:\PP^8\dashrightarrow Z_2\subset\PP^9$ 
of the map $\alpha$ with the parameterization $\rho$ of the quadric $Z_2$. 
One also see that the base locus 
of the Cremona transformation $\alpha$  is 
the union of the base locus of $\varphi_2$ 
with a simple point. 

\paragraph*{}
For $j=1$,
we use  Algorithm \ref{baseLocusOfInverseMap-algo}  
in order to compute 
a rational map $\varphi_1:\PP^8\dashrightarrow \PP^{10}$ with 
the same base locus of $\eta_1^{-1}$, but different from it.
Once this is done,
we can determine the image and the inverse of $\varphi_1$ 
(i.e. we correct $Z_1$ and $\eta_1$), by using 
Algorithms \ref{invertBirationalMapRS-algo} and \ref{homogPartOfImage-algo}.

\paragraph*{}
Finally, the case when $j=0$ is very special 
and can be obtained, for example,  as follows.
Consider the rational map 
$\widetilde{\eta_0}:\PP^{11}\dashrightarrow\PP^{11}$ 
defined by the linear 
system generated by the $9$ quadrics defining $\eta_0$ and 
the $3$ quadrics defining $Z_0$.
It turns out that $\widetilde{\eta_0}$ is a quadro-quadric Cremona 
transformation, and its inverse 
extends the map $\psi_0:\PP^7\dashrightarrow\PP^{11}$
 to a map of $\PP^{11}$. Thus we take $\varphi_0$ to be 
 the restriction to a general $\PP^8\subset\PP^{11}$ of $\widetilde{\eta_0}^{-1}$. 

\subsection{Proof of Theorem \ref{theoremscrolls}}\label{sec proof}
 We now must check that 
   the above constructed rational maps $\varphi_j$ 
   (whose data are stored in the file \verb!equationsOfBj.m2!)
    provide 
 the sought transformations.
From the main result of 
\citet{note2}, 
one immediately sees that the statement of Theorem~\ref{theoremscrolls} is equivalent 
to the following:
\emph{For $j=0,1,2,3$, there exists a special
quadratic birational transformation 
from $\PP^8$ into a factorial complete intersection of $3-j$ quadrics 
in $\PP^{11-j}$, 
whose base locus 
  is  a threefold 
 of degree $9+j$. }
 Moreover, a complete intersection of $3-j$ quadrics 
in $\PP^{11-j}$ is automatically factorial whenever its singular locus has dimension 
at most $4$ (see Subsection \ref{subsec: factoriality}).
Now the proof is completely reduced to
 computational verifications;
see Appendix~\ref{appendix A} for more details.
 
\subsection{Applications}\label{subsec: applications}
Here we list consequences of Theorem \ref{theoremscrolls}
related to the problem of classification of special birational transformations.
\subsubsection{Case $j=3$}  From Theorem \ref{theoremscrolls}, it follows  the existence of 
 a linearly normal  threefold 
 embedded in $\PP^8$ as a linear scroll 
 of degree $12$ over a rational surface $Y$ with $K_{Y}^2=5$.  
 As a consequence \citep[see][Theorem~7.1]{note2}, we deduce that there are exactly two types
of special quadratic Cremona transformations having three-dimensional base locus. 
In the other example, due to \citet{hulek-katz-schreyer},
the base locus is given by an internal projection 
 in $\PP^8$ of a general three-dimensional linear section  $\GG(1,5)\cap\PP^9\subset\PP^9$
 of  $\GG(1,5)\subset\PP^{14}$.
 Furthermore, from results in \citet{russo-qel1}, 
 it follows immediately that
  these two transformations  are also the only special
Cremona transformations of type $(2,5)$. See cases with $a=0$ in Table \ref{table: class}.
\subsubsection{Case $j=2$}\label{jequal2}  
The existence of a threefold embedded in $\PP^8$ as a linear scroll 
 of degree $11$ 
over the surface $\FF_1$  
 has been established by \citet[see also \citep{besana-fania-flamini-f1}]{alzati-fania-ruled}.  
 From Theorem~\ref{theoremscrolls}, we also see 
 that such a scroll is cut out by quadrics.  
 This is already sufficient to imply the existence 
 of a quadratic birational map 
 from $\PP^8$ into a quadric 
  $Q^8$ of $\PP^9$  \citep[see][Example~6.14]{note2}; 
 but, again from Theorem \ref{theoremscrolls},
 we see that there exists such an example with 
 $Q^8\subset\PP^9$ factorial.
In particular, 
we have solved the open problem left in
 \citet[Section~7]{note}.  
 That is, we can say that
 there are, 
 exactly, four types of 
  special birational transformations of type $(2,d)$, $d>1$,
  into a factorial hypersurface and whose 
base locus has dimension three; 
these are the cases with  $a=1$ in Table~\ref{table: class}.
(For the description of the
transformations of type $(2,1)$ 
into a hypersurface, see e.g. \citet[Section~3]{note}.)

\subsubsection{Cases $j=0$ and $j=1$}  The existence of a threefold 
embedded in $\PP^8$ as a linear scroll 
of degree $9$ over $\PP^2$, as well as 
  of degree $10$ over the quadric surface, 
has been established by   
\citet{fania-livorni-ten}.
 In these cases, one can also deduce
 directly  
 that they are cut out by quadrics and 
 that give birational maps 
 \citep[see][Examples~6.16 and 6.17]{note2}. 
 From Theorem~\ref{theoremscrolls},
 we see that there exist such maps with factorial image.
 In particular, the existence of the example in the case $j=1$, 
 together with Remark \ref{remark: factoriality},
 establishes the effectiveness of 
 the list of potential 
  special birational transformations 
 of type $(2,3)$ into a factorial 
 del Pezzo 
 variety, which is obtained in
 \citet[Corollary~7.2]{note2}; there are exactly three types 
 of such transformations 
 and correspond to the three cases with $d=3$ in Table~\ref{table: class}.
\subsubsection{Summary table}  For the convenience of the reader, 
 we summarize in Table \ref{table: class} 
 the list 
 of the  special birational transformations 
          $\varphi:\PP^n\dashrightarrow Z\subseteq\PP^{n+a}$ of type 
          $(2,d)$, with $d>1$, and whose base locus $\B$ has dimension $3$.
          For details and 
          the complete list 
          when 
          $\dim \B\leq 3$ and 
          including the case when $d=1$, 
          see \citet[Theorem~7.1]{note2}.        
          \emph{Notation}: in the Table, we denote by $Q^m$ an $m$-dimensional quadric hypersurface and
          by $Q^m_1\cap\cdots \cap Q^m_s$ a complete intersection 
          of $s$ $m$-dimensional quadric hypersurfaces; 
          $\lambda$ and $g$ respectively denote degree and sectional genus of $\B$.
 {\renewcommand{\arraystretch}{1.4}  
 \begin{table}[htbp]
\centering
\tabcolsep=4.8pt 
\begin{tabular}{cccclccl}  
\hline
$n$ & $a$ & $\lambda$ &  $g$  &  Abstract structure of $\B$ & $d$ & $Z$  & Existence   \\  
\hline
\hline
 $7$ & $1$ & $6$ & $1$  & Hyperplane section of $\PP^2\times\PP^2\subset\PP^8$ & $2$ & $Q^7$ & yes  \\ 
\rowcolor{Gray}
 \hline 
 $8$ & $0$ & $12$ & $6$  & Scroll over rational surf. with $K^2=5$   & $5$ & $\PP^8$  &  yes ($j=3$ in Thm. \ref{theoremscrolls}) \\
\hline 
 $8$ & $0$ & $13$ & $8$    &  Blow-up of a Mukai variety at a point & $5$ & $\PP^8$ & yes \citep{hulek-katz-schreyer}  \\
\hline 
 $8$ & $1$ &  $11$   &  $5$  &  Blow-up of $Q^3$ at $5$ points  & $3$ & cubic hyp. & yes \citep{note} \\
\rowcolor{Gray}
\hline 
 $8$ & $1$ &  $11$   &  $5$  & Scroll over $\FF_1$   & $4$ & $Q^8$ & yes ($j=2$ in Thm. \ref{theoremscrolls}) \\
\hline 
 $8$ & $1$ &  $12$   &  $7$ & Linear section of the spin. $\mathbb{S}^{10}\subset\PP^{15}$  & $4$ & $Q^8$ &  yes \\ 
\rowcolor{Gray}
\hline 
 $8$ & $2$ & $10$ & $4$ & Scroll over $Q^2$ & $3$ & $Q^9_1\cap Q^9_2$  &  yes ($j=1$ in Thm. \ref{theoremscrolls}) \\
\rowcolor{Gray}
\hline 
 $8$ & $3$ & $9$ & $3$ & Scroll over $\PP^2$ & $2$ & $Q^{10}_1\cap Q^{10}_2 \cap Q^{10}_3$  &  yes ($j=0$ in Thm. \ref{theoremscrolls}) \\
\hline 
 $8$ & $3$ & $9$ & $3$ & Quadric fibration over $\PP^1$ & $3$ & (*) &  yes (Remark \ref{remark: factoriality}) \\
\hline 
 $8$ &$4$ & $8$ & $2$ & Hyperplane section of $\PP^1\times Q^3$ & $2$ & (**) &   ? \ \  (yes if $Z$ is factorial)   \\
\hline 
 $8$ &$6$ & $6$ & $0$ & Rational normal scroll & $2$ & $\GG(1,5)$ & yes  \\  
\hline 
\end{tabular}
 \caption{All the special birational transformations 
          $\PP^n\dashrightarrow Z\subseteq\PP^{n+a}$ of type 
          $(2,d)$ with 
           $d>1$ and $\dim \B=3$.
             (*): $Z$ is the cone in $\PP^{11}$ over $\GG(1,4)\subset\PP^9$; 
            (**): $Z$ is a quadric section of the cone in $\PP^{12}$ over $\GG(1,4)\subset\PP^9$.
 }
\label{table: class} 
\end{table}
}
 \begin{rem}\label{remark: factoriality}
  The third last case in Table \ref{table: class} is of a special birational 
  transformation $\varphi:\PP^8\dashrightarrow Z\subset\PP^{11}$ 
  of type $(2,3)$ whose base locus is a quadric fibration of degree $9$.
  Such a transformation exists  \citep[see][Example~6.17]{note2}, 
  and $Z$ is the cone in $\PP^{11}$ 
  over the Grassmannian $\mathbb{G}(1,4)\subset\PP^9\subset\PP^{11}$ of 
  vertex a line $L\subset\PP^{11}\setminus\PP^9$.
  We now deduce that $Z$ is a factorial variety from Proposition~\ref{proposition: factoriality}.   
 \end{rem}
 \section{Two examples of quadro-quadric Cremona transformations}\label{sec: cremona}
 In this section, we apply again the algorithms described in Section \ref{sec: methods} to 
 obtain two explicit examples of Cremona transformations of type $(2,2)$.  
 Their appropriate restrictions yield explicit examples of special 
 birational transformations of type $(2,2)$ into complete intersections of quadrics. 
 The 
 code here used is contained in the file \verb!cremona22-input.m2!.
 
 \subsection{A Cremona transformation of \texorpdfstring{$\PP^{11}$}{P11}}\label{subsec: CremonaP11}
 We have already encountered this example 
 in the discussion of the case $j=0$ in Subsubsection~\ref{subsec: liftings}. 
 We return to the construction in more details and
 slightly simplifying the equations.
  Consider the set $\Lambda$ of the $7$ points in $\PP^2$:
  $[1,0,0]$, $[0,1,0]$, $[0,0,1]$, $[1,1,0]$, $[1,0,1]$, $[0,1,1]$, $[1,1,1]$. 
The saturate ideal of $\Lambda$ is generated by $3$ cubics and 
$\dim |\I_{\Lambda\subset\PP^2}(4)|=7$.
The image $Y\subset\PP^7$ of the rational map 
$\varphi=\varphi_{|\I_{\Lambda\subset \PP^2}(4)|}:\PP^2\dashrightarrow\PP^7$  
 is a smooth surface 
of degree  $9$, and its saturated ideal is generated by $12$ quadrics.
These $12$ quadrics define a 
special birational map $\psi:\PP^7\dashrightarrow\PP^{11}$  of type $(2,2)$
into a complete intersection of $4$ quadrics. 
We then consider the $8$ quadrics defining 
$\psi^{-1}$ together to the $4$ quadrics defining $\overline{\psi(\PP^7)}$.
These $12$ quadrics give a Cremona transformation of $\PP^{11}$.
Explicitly, there is a 
Cremona transformation $\omega:\PP^{11}\dashrightarrow\PP^{11}$ 
given by:
\begin{equation}\label{eqsCremonaP11}
\renewcommand{\arraystretch}{1.0}
\begin{array}{c} 
x_6\,x_{10}-x_5\,x_{11}, \ 
x_1\,x_{10}-x_4\,x_{10}+x_3\,x_{11}, \ 
x_6\,x_8-x_2\,x_{11}, \ 
x_5\,x_8-x_2\,x_{10}, \\ 
x_3\,x_8-x_0\,x_{10}, \ 
x_1\,x_8-x_4\,x_8+x_0\,x_{11}, \  
x_6\,x_7-x_1\,x_9+x_4\,x_9-x_4\,x_{11}, \  
x_5\,x_7+x_3\,x_9-x_4\,x_{10}, \\ 
x_2\,x_7-x_4\,x_8+x_0\,x_9, \  
x_1\,x_5-x_4\,x_5+x_3\,x_6, \ 
x_2\,x_3-x_0\,x_5, \ 
x_1\,x_2-x_2\,x_4+x_0\,x_6 .
\end{array}
\end{equation}
The inverse $\omega^{-1}$ is given by:
\begin{equation}\label{eqsCremonaP11INV}
\renewcommand{\arraystretch}{1.0}
\begin{array}{c}
-x_5\,x_{10}+x_4\,x_{11}, \ 
x_5\,x_9-x_8\,x_9+x_6\,x_{10}-x_1\,x_{11}+x_7\,x_{11}, \  
x_2\,x_{10}+x_3\,x_{11}, \  
x_4\,x_9-x_1\,x_{10}, \\  
-x_8\,x_9+x_6\,x_{10}+x_7\,x_{11}, \ 
x_3\,x_9+x_0\,x_{10}, \ 
x_2\,x_9-x_0\,x_{11}, \  
x_4\,x_6+x_5\,x_7-x_1\,x_8, \\  
x_2\,x_4+x_3\,x_5, \  
-x_3\,x_6+x_2\,x_7-x_0\,x_8, \  
x_1\,x_3+x_0\,x_4, \  
x_1\,x_2-x_0\,x_5  .
\end{array}
\end{equation}   
The base locus $X$ of $\omega$ 
has
dimension $6$,
degree $9$, and Hilbert polynomial 
$$
\renewcommand{\arraystretch}{1.0}
P(t)=3\,\begin{pmatrix}t+4\cr 4\end{pmatrix}−11\,\begin{pmatrix}t+5\cr 5\end{pmatrix}+9\,\begin{pmatrix}t+6\cr 6\end{pmatrix}.
$$
The support of the singular locus of $X$ is a linear space of dimension $2$. 
Moreover, $X$ is absolutely irreducible. 
Indeed, we can obtain a birational map $\gamma:\PP^6\dashrightarrow X\subset\PP^{11}$  
of type $(2,1)$, 
imitating a known construction for smooth Severi varieties 
\citep[that is, a smooth Severi variety $\mathbf{S}$ can be parameterized
by   the inverse of the restriction to $\mathbf{S}$ 
 of the linear projection from the linear span 
of the general entry locus; see][Theorem~2.4(f)]{zak-tangent}.

\subsection{A Cremona transformation of \texorpdfstring{$\PP^{20}$}{P20}}
Similarly to the previous construction, 
we now
construct  an example
 of Cremona transformation of $\PP^{20}$.   
Let $E\subset\PP^7\subset\PP^8$ be a degenerate 
$3$-dimensional Edge variety of degree $7$; namely, $E$ is
the residual intersection of $\PP^1\times\PP^3\subset\PP^7$ with
a general quadric in $\PP^7$ containing one of the $\PP^3$'s of the 
rulings of $\PP^1\times\PP^3\subset\PP^7$. 
Denoting by $t_0,\ldots,t_8$ homogeneous coordinates on $\PP^8$, 
we can take $\PP^7\subset\PP^8$ to be the hyperplane defined by $t_8$, 
and $E\subset\PP^7$ to be defined 
 by the following $8$ 
quadratic forms \citep[see][Example~5.2]{note}:   
\begin{equation}\label{edge37}
\renewcommand{\arraystretch}{1.0}
\begin{array}{c}
-t_1\,t_4+t_0\,t_5, \ 
-t_2\,t_4+t_0\,t_6, \ 
-t_3\,t_4+t_0\,t_7, \ 
        -t_2\,t_5+t_1\,t_6, \ 
        -t_3\,t_5+t_1\,t_7, \  
        -t_3\,t_6+t_2\,t_7, \\ 
        t_0^2+t_2^2+t_3^2+t_1\,t_5+t_1\,t_6+t_0\,t_7, \ 
        t_0\,t_4+t_5^2+t_2\,t_6+t_5\,t_6+t_3\,t_7+t_4\,t_7  .
\end{array}
\end{equation}
Now, consider the birational map $\varphi:\PP^8\dashrightarrow Y\subset \PP^{16}$ of type $(2,1)$ 
defined by the linear system $|\I_{E\subset \PP^8}(2)|$; namely, $\varphi$ is defined by the quadrics 
in (\ref{edge37}) together to the monomials $t_0t_8,\ldots,t_8^2$.
The image $Y$ of $\varphi$ is an irreducible variety of degree $33$, and its saturated ideal 
is generated by $21$ quadrics.
We then consider the rational map $\psi:\PP^{16}\dashrightarrow Z\subset \PP^{20}$ 
defined by these $21$ quadrics, where $Z$ denotes its image.
Here the calculations become 
much more arduous,
and we are unable to determine the ideal of $Z$.  
However, via Algorithm \ref{homogPartOfImage-algo}, 
one can  determine the ideal generated by all quadrics containing $Z$, 
and then to see that it is generated by $4$ quadrics.
Now, 
applying Algorithm \ref{invertBirationalMapRS-algo}, 
by passing as input $\psi$ and the ideal generated 
by these $4$ quadrics,
one obtains (after several hours) 
that 
  $\psi$ is birational with inverse defined by quadrics.
Finally, by taking together the $17$ quadrics defining $\psi^{-1}$, 
and the $4$ independent quadrics containing $Z\subset\PP^{20}$, 
we obtain a 
Cremona transformation $\omega:\PP^{20}\dashrightarrow\PP^{20}$, explicitly given by: 
\begin{equation}\label{eqsCremonaP20}
\renewcommand{\arraystretch}{1.0}
\begin{array}{c}
x_{10}\,x_{15}-x_9\,x_{16}+x_6\,x_{20}, \ 
 x_{10}\,x_{14}-x_8\,x_{16}+x_5\,x_{20}, \ 
 x_9\,x_{14}-x_8\,x_{15}+x_4\,x_{20}, \\
 x_6\,x_{14}-x_5\,x_{15}+x_4\,x_{16}, \ 
 x_{11}\,x_{13}-x_{16}\,x_{17}+x_{15}\,x_{18}-x_{14}\,x_{19}+x_{12}\,x_{20}, \\ 
 x_3\,x_{13}-x_{10}\,x_{17}+x_9\,x_{18}-x_8\,x_{19}+x_7\,x_{20}, \
 x_{10}\,x_{12}-x_2\,x_{13}-x_7\,x_{16}-x_6\,x_{18}+x_5\,x_{19}, \\ 
 x_9\,x_{12}-x_1\,x_{13}-x_7\,x_{15}-x_6\,x_{17}+x_4\,x_{19}, \ 
 x_8\,x_{12}-x_0\,x_{13}-x_7\,x_{14}-x_5\,x_{17}+x_4\,x_{18}, \\
 x_{10}\,x_{11}-x_3\,x_{16}+x_2\,x_{20}, \  
 x_9\,x_{11}-x_3\,x_{15}+x_1\,x_{20}, \ 
 x_8\,x_{11}-x_3\,x_{14}+x_0\,x_{20}, \\ 
 x_7\,x_{11}-x_3\,x_{12}+x_2\,x_{17}-x_1\,x_{18}+x_0\,x_{19}, \ 
 x_6\,x_{11}-x_2\,x_{15}+x_1\,x_{16}, \ 
 x_5\,x_{11}-x_2\,x_{14}+x_0\,x_{16}, \\
 x_4\,x_{11}-x_1\,x_{14}+x_0\,x_{15}, \ 
 x_6\,x_8-x_5\,x_9+x_4\,x_{10}, \ 
 x_3\,x_6-x_2\,x_9+x_1\,x_{10}, \\
 x_3\,x_5-x_2\,x_8+x_0\,x_{10}, \ 
 x_3\,x_4-x_1\,x_8+x_0\,x_9, \ 
 x_2\,x_4-x_1\,x_5+x_0\,x_6 .
  \end{array}
\end{equation}
The inverse $\omega^{-1}$ is given by: 
\begin{equation}\label{CremonaY}
\renewcommand{\arraystretch}{1.0}
\begin{array}{c}
-x_{15}\,x_{18}+x_{14}\,x_{19}-x_{11}\,x_{20}, \ 
 -x_{15}\,x_{17}+x_{13}\,x_{19}-x_{10}\,x_{20}, \ 
 -x_{14}\,x_{17}+x_{13}\,x_{18}-x_9\,x_{20}, \\
 -x_{11}\,x_{17}+x_{10}\,x_{18}-x_9\,x_{19}, \ 
 -x_{15}\,x_{16}+x_3\,x_{19}-x_2\,x_{20}, \ 
 -x_{14}\,x_{16}+x_3\,x_{18}-x_1\,x_{20}, \\
 -x_{13}\,x_{16}+x_3\,x_{17}-x_0\,x_{20}, \ 
 -x_{12}\,x_{16}-x_8\,x_{17}+x_7\,x_{18}-x_6\,x_{19}-x_5\,x_{20}, \\ 
 -x_{11}\,x_{16}+x_2\,x_{18}-x_1\,x_{19}, \ 
 -x_{10}\,x_{16}+x_2\,x_{17}-x_0\,x_{19}, \ 
 -x_9\,x_{16}+x_1\,x_{17}-x_0\,x_{18}, \\
 -x_{11}\,x_{13}+x_{10}\,x_{14}-x_9\,x_{15}, \ 
 -x_3\,x_{12}-x_8\,x_{13}+x_7\,x_{14}-x_6\,x_{15}-x_4\,x_{20}, \\
 x_3\,x_5+x_2\,x_6-x_1\,x_7+x_0\,x_8-x_4\,x_{16}, \
 -x_3\,x_{11}+x_2\,x_{14}-x_1\,x_{15}, \\
 -x_3\,x_{10}+x_2\,x_{13}-x_0\,x_{15}, \ 
 -x_3\,x_9+x_1\,x_{13}-x_0\,x_{14}, \\
 -x_8\,x_{10}+x_7\,x_{11}-x_2\,x_{12}+x_5\,x_{15}-x_4\,x_{19}, \ 
 -x_8\,x_9+x_6\,x_{11}-x_1\,x_{12}+x_5\,x_{14}-x_4\,x_{18}, \\
 -x_7\,x_9+x_6\,x_{10}-x_0\,x_{12}+x_5\,x_{13}-x_4\,x_{17}, \ 
 -x_2\,x_9+x_1\,x_{10}-x_0\,x_{11} . 
  \end{array}
\end{equation}     
The base locus $X$ of $\omega$ 
has  dimension $12$, degree $33$, 
and Hilbert polynomial 
$$
\renewcommand{\arraystretch}{1.0}
P(t)=\begin{pmatrix}t+8\cr 8\end{pmatrix}−12\,\begin{pmatrix}t+9\cr 9\end{pmatrix}+45\,\begin{pmatrix}t+10\cr 10\end{pmatrix}−66\,\begin{pmatrix}t+11\cr 11\end{pmatrix}+33\,\begin{pmatrix}t+12\cr 12\end{pmatrix}.
$$  
As in the case of Subsection \ref{subsec: CremonaP11}, 
we can obtain 
a parameterization of 
 $X$  with a birational map $\gamma:\PP^{12}\dashrightarrow X\subset\PP^{20}$
of type $(2,1)$; one of these parameterizations is stored in the file \verb!cremona22.m2!. In particular,
$X$ is absolutely irreducible.    
We 
are not able to determine the singular locus of $X$,
but one can see that the vertex of the secant variety
of $X$ is a linear space of 
dimension $5$, and it is contained in the singular locus of $X$. 
\subsection{Some remarks about related open problems}
After the classification results  
 on  special quadro-quadric birational
transformations, into  a projective space \citep{ein-shepherdbarron},
  and into a quadric hypersurface \citep{note},
one can  consider the possible next problem of classifying special quadro-quadric
 birational transformations $\varphi:\PP^n\dashrightarrow Z\subset\PP^{n+c}$ 
 into (factorial) complete intersections 
of $c\geq 2$ quadric hypersurfaces. 
Using general results about special  
birational transformations, 
and results of the theory of $(L)QEL$-manifolds contained in 
\citep{russo-qel1,ionescu-russo-conicconnected}, 
one can easily deduce that,
for $2\leq c\leq 5$, putting $r=\dim \B$ the dimension of the base locus $\B$ of $\varphi$,
one of the following cases holds: 
\begin{enumerate}[(i)]
\item\label{P2P2} $r=2$, $n=6$, $c=2$;
\item\label{cremonaP11} $r=3$, $n=8$, $c=3$; 
\item\label{G15} $r=8-c$, $n=14-c$, $c\leq4$, and if $c\leq 3$, then $\B$ is a prime Fano manifold of coindex $3$;
\item\label{cremP17}  $r=5$, $n=12$, $c=5$;  
\item\label{cremonaP20} $r=12-c$, $n=20-c$, $\B$ is a prime Fano manifold of coindex $4$;
\item\label{E6} $r=16-c$, $n=26-c$, $\B$ is a prime Fano manifold of coindex $5$.  
\end{enumerate}
In the cases (\ref{P2P2}) and (\ref{cremonaP11}),
since $r\leq3$,
we apply the results of \citet{note2};
so we deduce that, in case (\ref{P2P2}), 
$\B$ is  a linear section of 
$\PP^2\times\PP^2\subset\PP^8$, and
 in case (\ref{cremonaP11}),  $\B$ is a threefold scroll of 
 degree $9$ over $\PP^2$. 
 In case (\ref{G15}) with $c\leq 3$, 
 by applying results of 
  \citet{mukai-biregularclassification}, 
we see that $\B$ is a linear section of $\GG(1,5)\subset\PP^{14}$.
In cases (\ref{cremonaP20}) and (\ref{E6}), 
using the same argument of \citet[Proposition~4.4(3)]{note}, 
we can determine the Hilbert polynomial 
of $\B$, and so we see that it is, respectively, as that 
of a general linear section 
of the scheme defined by (\ref{eqsCremonaP20}), 
and as that 
of a general 
linear section of $E_6\subset\PP^{26}$; 
in particular $\deg \B$ is, respectively, $33$ and $78$.
Finally, case (\ref{cremonaP20}) with $c=2$  is excluded 
from the main results in \citep{mok,russo-qel1}. 
Summarizing, in particular, we have the following: 
\begin{prop}
 Keeping notation as above, if $0\leq c\leq 3$, then 
 $\B$ has the same Hilbert polynomial as
 a general $c$-codimensional linear section 
 of the base locus of one of the following quadro-quadric Cremona transformations:
 \begin{enumerate}[(a)]
  \item one of the four special quadro-quadric Cremona transformations;
  \item one of the two transformations  defined by (\ref{eqsCremonaP11}) and (\ref{eqsCremonaP20}) (here $c=3$).
 \end{enumerate}
\end{prop} 

\appendix
\section{Computational details for the proof of Theorem~\ref{theoremscrolls}}\label{appendix A}
In this appendix, we explain how one can deduce 
Theorem~\ref{theoremscrolls} from the data of the file
\verb!equationsOfBj.m2!, which contains the output of the
procedure described in Subsection~\ref{sec construction}.

Loading this file in {\sc Macaulay2}
 produces several groups of 
 variables  indexed by an integer $j$ from $0$ to $3$:
 \begin{itemize}
\item \verb!phi_j! is a ring map 
 corresponding to a quadratic rational map $\varphi_j:\PP^8\dashrightarrow\PP^{11-j}$;
\item \verb!idealB_j! is simply defined as the ideal of the base locus of $\varphi_j$;
\item \verb!Z_j! is the ideal of 
           a  complete intersection $Z_j\subseteq \PP^{11-j}$ of $3-j$ quadrics;
\item \verb!eta_j!    is a ring map 
 corresponding to a rational map $\eta_j:\PP^{11-j}\dashrightarrow\PP^{8}$
 defined by forms of degree $2+j$.
\end{itemize}
In order to give a glimpse of what
 this file contains, 
we show the 
number of characters for writing the main data:
{\footnotesize
\begin{verbatim} 
i1 : load "equationsOfBj.m2"
i2 : -- number characters for writing phi_j
     for j to 3 list # toString phi_j 
o2 = {1668, 16044, 18880, 26486}
i3 : -- number characters for writing eta_j 
     for j to 3 list # toString eta_j 
o3 = {1231, 160297, 1409699, 1702444} 
\end{verbatim}
} \noindent 
 It is quite standard
 and fast
 to check that 
 \verb!idealB_j! is the saturated homogeneous ideal of 
 a threefold $\B_j\subset\PP^8$ of degree $9+j$, 
 as well as computing the dimension of the singular locus of $Z_j$,
 and so to see that $Z_j$ is factorial. 
 Further, using a tool like  \verb!topComponents!, 
 one sees that $\B_j$ is equidimensional.
 We conclude 
 by checking that $\varphi_j$ is birational, 
 $\B_j$ is smooth and absolutely connected, and the image of $\varphi_j$ is $Z_j$. 
 \subsection{Birationality of \texorpdfstring{$\varphi_j$}{phij}}
 We see that $\varphi_j$ is birational, by checking that $\eta_j$ is a lifting 
 to $\PP^{11-j}$ of the inverse of $\varphi_j$, i.e. that $\eta_j\circ\varphi_j$ 
 coincides, as a rational map, with the identity of $\PP^8$ (Algorithm~\ref{isInverseMap-algo}):
 {\footnotesize 
\begin{verbatim}
i4 : load "bir.m2"; 
i5 : for j to 3 list time isInverseMap(matrix phi_j,matrix eta_j)
     -- used 0.0944848 seconds
     -- used 78.4827 seconds
     -- used 1878.9 seconds
     -- used 31155.8 seconds 
o5 = {true, true, true, true}
\end{verbatim}
} \noindent 
By the way of illustration, we use 
Algorithm \ref{degreeOfRationalMap-algo} 
to check again birationality:
 {\footnotesize 
\begin{verbatim}
i6 : for j to 3 list time degreeOfRationalMap(matrix phi_j)
     -- used 0.212941 seconds
     -- used 0.47604 seconds
     -- used 0.790374 seconds
     -- used 2.23214 seconds 
o6 = {1, 1, 1, 1} 
\end{verbatim}
} \noindent 
Further, we determine all the projective degrees of $\varphi_j$ (see Remark~\ref{remark: projectiveDegrees}):
 {\footnotesize 
\begin{verbatim}
i7 : time projectiveDegrees(matrix phi_0,ideal ringP8)
     -- used 697.504 seconds 
o7 = {8, 16, 23, 23, 16, 8, 4, 2, 1} 
i8 : time projectiveDegrees(matrix phi_1,ideal ringP8)
     -- used 2209.8 seconds 
o8 = {4, 12, 20, 22, 16, 8, 4, 2, 1} 
i9 : time projectiveDegrees(matrix phi_2,ideal ringP8)
     -- used 10434.6 seconds
o9 = {2, 8, 17, 21, 16, 8, 4, 2, 1}
i10 : time projectiveDegrees(matrix phi_3,ideal ringP8)
      -- used 3477.02 seconds 
o10 = {1, 5, 14, 20, 16, 8, 4, 2, 1} 
\end{verbatim}
} \noindent 

 \subsection{Regularity of \texorpdfstring{$\B_j$}{Bj}}
We check the absolute connection of $\B_j$ using Algorithm \ref{numberConnectedComponents-algo}:
{\footnotesize 
\begin{verbatim}
i11 : for j to 3 list time numberConnectedComponents idealB_j 
      -- used 1.08971 seconds
      -- used 60.295 seconds
      -- used 290.407 seconds
      -- used 12015.7 seconds 
o11 = {1, 1, 1, 1} 
\end{verbatim}
} \noindent 
Now, in order to prove that the scheme $\B_j$ is smooth, 
it is sufficient to show that it is smooth as a scheme over $\ZZ/(p)$,
where $p$ is some prime 
that does not produce a lowering of the codimension (see  Subsection \ref{smoothnessBaseLocus}).
A good prime that we can take is for example $113$, and 
this is verifiable through Sage in a few minutes:
{\footnotesize 
\begin{verbatim}
i12 : for j to 3 list isSmooth(idealB_j,reduceToChar=>113,Use=>Sage) 
o12 = {true, true, true, true} 
\end{verbatim}
} \noindent 
Alternately,  
we have also 
provided the file
\verb!Bjs.sage!, 
which contains the definitions of the schemes $\B_j$'s
in the Sage language. 
So we can check the smoothness of the $\B_j$'s 
by using directly Sage (for instance $j=3$):
{\footnotesize
\begin{verbatim}
sage: kk=GF(113);
sage: load('Bjs.sage') 
sage: SingB3=B3.intersection(P8.subscheme(jacobian(eqsB3,varsP8).minors(5)));
sage: SingB3.dimension() 
-1
sage: exit
Exiting Sage (CPU time 1m47.59s, Wall time 2m14.86s).
\end{verbatim}
} \noindent 
\subsection{Image of \texorpdfstring{$\varphi_j$}{phij}}\label{Subsection appendix Image phij}
Finally, we compute the image of $\varphi_j$.
Running the following code also shows how 
Algorithm~\ref{homogPartOfImage-algo} turns out to be much faster than computing  
the kernel of the ring map \verb!phi_j!. 
{\footnotesize
\begin{verbatim}
i13 : time Z0 = homogPartOfImage(phi_0,2);
      -- used 3.93689 seconds
i14 : time Z0'= kernel phi_0;
      -- used 3142.35 seconds
i15 : Z0' == ideal Z0 and Z0' == Z_0
o15 = true
i16 : time Z1 = homogPartOfImage(phi_1,2);
      -- used 2.83038 seconds
i17 : time Z1'= kernel phi_1;
      -- used 15228.3 seconds
i18 : Z1' == ideal Z1 and Z1' == Z_1
o18 = true
i19 : time Z2 = homogPartOfImage(phi_2,2);
      -- used 1.30964 seconds
i20 : time Z2'= kernel phi_2;
     -- used 14754.2 seconds
i21 : Z2' == ideal Z2 and Z2' == Z_2
o21 = true 
\end{verbatim}
} \noindent

\begin{ack}
I wish to thank
the Department of Computer Science    
of the 
Federal University of Rio de Janeiro 
for allowing me to access
to the supercomputer where
the codes of this paper were run.
I also wish to thank Nivaldo Medeiros and Abramo Hefez
for stimulating discussions, and 
Francesco Russo 
for introducing me to the fascinating subject 
of special birational transformations. 
A special thanks also goes to 
the associate editor and the anonymous referees
 of the Journal of Symbolic Computation
for their helpful comments.
\end{ack}


\end{document}